\documentclass{amsart}
\usepackage{amssymb}
\usepackage{mathrsfs}
\usepackage{graphicx}
\usepackage{comment}
\newcommand{\comments}[1]{}
\numberwithin{equation}{section}

\newcommand{\be}{\begin{equation}}
\newcommand{\ee}{\end{equation}}
\newcommand{\ba}{\begin{align}}
\newcommand{\ea}{\end{align}}

\newcommand{\abs}[1]{\lvert#1\rvert}

\newtheorem{example}{Example}[section]
\newtheorem{theorem}{Theorem}[section]

\newtheorem{conjecture}{Conjecture}[section]
{\begin{list}{}{%
\settowidth{\labelwidth}{\textsf{{\it #1.}}}%
\setlength{\labelsep}{4mm}%
\setlength{\leftmargin}{\labelwidth}%
\addtolength{\leftmargin}{\labelsep}%
}}%
{\end{list}}

{\begin{list}{}{%
\settowidth{\labelwidth}{\textsf{{\it #1.}}}%
\setlength{\labelsep}{2mm}%
\setlength{\leftmargin}{\labelwidth}%
\addtolength{\leftmargin}{\labelsep}%
\addtolength{\leftmargin}{4mm}%
\setlength{\itemsep}{6pt}%
\setlength{\listparindent}{0pt}%
\setlength{\topsep}{3pt}%
}}%
{\end{list}}

\title{A generalization of the Goresky-Klapper conjecture, Part II}

\author[T. Cochrane]{Todd Cochrane}
\address{ Department of Mathematics\\
         Kansas State University\\
         Manhattan, KS 66506 USA}
\email{cochrane@ksu.edu, pinner@ksu.edu, crichardson@ksu.edu}

\author[M. Mossinghoff]{Michael J. Mossinghoff}\thanks{This work was supported in part by a grant from the Simons Foundation (\#426694 to M.~J. Mossinghoff).}
\address{Department of Mathematics \&  Computer Science \\Davidson College\\Davidson, NC 28035, USA. {\it E-mail address:} {\tt mimossinghoff@davidson.edu}}

\author[C. Pinner]{Chris Pinner}

\author[CJ Richardson]{C. J.  Richardson}

\date{\today}

\begin{document}

%\selectlanguage{francais}
%\renewcommand{\abstractname}{R\'{e}sum\'{e}}

\begin{abstract} Suppose that $f(x)=Ax^k$ mod $p$ is a permutation of the least residues mod $p$.
With the exception of the maps $f(x)=Ax$ and $Ax^{(p+1)/2}$
mod $p$ we show that for fixed $n\geq 2$  the image of each residue class mod $n$ contains elements from every  residue classe mod $n$,  once  $p$ is sufficiently large. If $f(x)=Ax$ mod $p$, then for each $p$ and $n$ there will be exactly $(1+o(1))\frac{6}{\pi^2}n^2$  readily describable
values of $A$ for which the image of some residue class mod $n$  misses at least one residue class mod $n,$ even when $p$ is large relative to $n$. A similar situation holds for $f(x)=Ax^{(p+1)/2}$ mod $p$.

\end{abstract}

\maketitle

\section{Introduction}\label{secIntroduction}

For an odd prime $p$ we let $I=\{1,2,\ldots ,p-1\}$ denote the reduced residues mod $p$, and  $f:I\rightarrow I$ a permutation of $I$  of the form
\be \label{form}  f(x)= Ax^k \text{ mod } p, \ee
with $A,k$ integers. Generally we assume that
\be \label{range}  |A| < p/2, \;\; p\nmid A,\;\;\; 1\leq k< p-1,\;\; \gcd(k,p-1)=1, \ee
although occasionally we allow $k$ to be negative  with $|k|< (p-1)/2$; $f(x)$ is determined by the value of $k$ mod $(p-1)$.

Goresky \& Klapper \cite{GK1} divided $I$ into the even and odd residues
$$ E=\{2,4,\ldots ,p-1\},\, \;\;\; O=\{1,3,\ldots , p-2\}, $$
and asked when $f$ could  also be  a permutation of $E$ (equivalently of  $O$).  Apart from the identity map $(p;A,k)=(p;1,1)$ they found six cases
$$ (p;A,k)=(5;-2,3), (7;1,5), (11;-2,3),(11;3,7), (11;5,9),(13;1,5), $$
and conjectured that there were no more for $p>13$. This was proved for sufficiently large $p$ in \cite{Bourgain} and in full in \cite{CochKony}, with asymptotic counts on $\abs{f(E)\cap O}$ considered in \cite{Bourgain2}. Since $x\mapsto p-x$ switches elements of $E$ and $O$,  this is the same as asking when $f(E)=O$ or $f(O)=E$, on replacing $A$ by $-A.$
A related question of Lehmer \cite[Problem F12, p. 381]{guy} asks how often  $x$ mod $p$ and its  inverse,  $f(x)=x^{-1}$ mod $p$,  have opposite parity; see Zhang \cite{Wenpeng}, or  the generalizations by Alkan, Stan and Zaharescu \cite{Alkan}, Lu and Yi \cite{luyi1,luyi2}, Shparlinski \cite{Igor,Igor2}, Xi and Yi \cite{xiyi}, and Yi and Zhang \cite{yizhang}.

Regarding  even and odd as a mod 2 property, we ask the same  question for a general modulus $n$. Dividing
 $I$ up  into the $n$ congruence classes mod $n$,
\begin{equation} \label{Ij}
I_j:=\{x \; :\; 1\leq x\leq p-1,\; x\equiv j \hbox{ mod }n \},\;\; j=0,\ldots ,n-1,
\end{equation}
there are now several different ways of generalizing the concept of a  permutation \eqref{form} having $f(O)=O$  and $f(E)=E$, or
$f(O)=E$ and $f(E)=O$. In \cite{CJGK1}
we identified five types of   $f(x)$:

%\begin{itemize}
{\bf Type (i):} $f(I_j)=I_j$  for all $j=0,\ldots, n-1.$

{\bf Type (iia):}  $f(I_0),\ldots ,f(I_{n-1})$ a permutation of $I_0,\ldots ,I_{n-1}$.

{\bf Type (iib):}   $f(I_j)=I_j$  for some  $j$.

{\bf Type (iii):}  There is a pair  $i,j$ with  $f(I_i)\subseteq I_j.$

{\bf Type (iv):}  There is a pair  $i,j$ with $f(I_i)\cap I_j=\emptyset. $

%\end{itemize}

Notice that for $n=2$ these are all the same problem, but for  general $n$ they can be quite different
 (indeed the $I_j$ may not even have the same cardinality).

In the first paper \cite{CJGK1} our focus was primarily on the Type (i)-(iii) maps, showing  that,
with the exception of $f(x)=\pm x$ mod $p$ when $n$ is even, and  $f(x)=\pm x$ or $\pm x^{(p+1)/2}$ mod $p$ when $n$ is odd, every $f(I_i)$ must contain elements from at least two different $I_j$ once
$ p\geq 9\cdot 10^{34} n^{92/3}$.

Here we are mainly interested in the Type (iv) maps.  When
\be \label{defd} d:=\gcd(p-1,k-1) \ee
is suitably small we showed in \cite{CJGK1} that the  values of $f(I_i)$ are, from an asymptotic point of view,  distributed equally  in the $n$ residue classes, ruling out any Type (iv) maps.
In particular we shall need the following result, Theorem 3.2, from
\cite{CJGK1}.

\begin{theorem} \label{smallgcd2} Let $p$ be an odd prime and $A,k,n$ integers satisfying \eqref{range} with $n\geq 2$,
and
$$ d=\gcd(k-1,p-1) \leq  0.006 p^{89/92}. $$

\noindent
 For any $i,j$, $0 \le i,j < n$, we have $f(I_i) \cap I_j \neq \emptyset $ provided that
$$
p>4 \cdot 10^{29}\; n^{\frac {184}3}.
$$

\end{theorem}

For small  $|k|$ this bound can be improved, for example  Theorem 1.1 of \cite{CJGK1}:

\begin{theorem} \label{lehmer} Suppose that $f(x) = Ax^{k}$  mod $p$, with $k\neq 1$, positive or negative.
If $p\geq 16.2 |k-1|^2 n^4$ then $f(x)$ is not a Type (iv) map.

\end{theorem}

Note here we  are thinking of $p>n^2$, otherwise any permutation $f(x)$ is a Type (iv) mapping. Indeed, if $p<n^2$ there will always  be a residue class $I_i$, and hence its image $f(I_i)$,  containing fewer than $n$ elements.

If  we want a stronger statement avoiding cases of  Type (iv) even when $d$ is large, that is, prove that the image of every residue class mod $n$  hits every residue class mod $n,$ then we will need to exclude more examples for $n>2$. For the linear maps,  $k=1$, we see in the next example that the image of each residue class mod $n$ will miss at least one residue class mod $n$ when the coefficient $A$ is sufficiently small, or more generally, of the form
\be \label{fracA} A=\frac{tp-r}{s},\hspace{3ex} \gcd(r,s)=1,\ee
for some integers $r,s,t$  with  $s\neq 0$, and  $r$ and $s$  sufficiently small. Note that any such representation also has $(t,s)=1$.

\begin{example}\label{Ex2} Suppose that $f(x)=Ax$ mod $p$ with $A$ an integer satisfying \eqref{range}.

\begin{itemize}
\item[(a)] If  $|A|<n$, or more generally,
\item[(b)] if $A$ is of the form \eqref{fracA} with  $|r|+|s|+ \gcd(n,s) -1 \leq n$,
\end{itemize}
then for each $i$ there
is at least one $j$ with $f(I_i)\cap I_j=\emptyset$.

\begin{itemize}
\item[(c)] If $A$ is of the form \eqref{fracA} with
\be \label{sharp} |r|+|s|\leq n \ee
\end{itemize}
then at least $n/\gcd(n,s)$ residue classes $I_i$ will have  $f(I_i)\cap I_j=\emptyset$ for some $j$.

Indeed, letting $B:=|A|$ in case (a),  $B:=|r|+|s|+\gcd(n,s)-2$ in case (b) and $B:=|r|+|s|-1$ in case (c), the number of missed residue classes $I_j$ will be at least $n-B$.
\end{example}
Note that (a) is a special case of (b) with $s=1,t=0,r=-A$, and (c) coincides with (b) when $\gcd(n,s)=1$.

A similar situation occurs for exponent $k=(p+1)/2,$ though we must  halve the range of restriction, as we see in the next example.

\begin{example}\label{Ex2b} Suppose that $p\equiv 1$ mod $4$ and  $f(x)=Ax^{(p+1)/2}$ mod $p$.

\noindent
If $A$ satisfies \eqref{range} and
\begin{itemize}
\item[(a)]  $2|A|<n$, or more generally,
\item[(b)] $A$ is of the form \eqref{fracA} with $2(|r|+|s|+\gcd(n,s)-2) <n$,
\end{itemize}
then for each $i$ there
is at least one $j$ with $f(I_i)\cap I_j=\emptyset$.
Indeed, if the restriction in parts (a) and (b) takes the form $B<n$ (as in the preceding example) then in each case the number of missed residue classes $I_j$  will be at least $n-B$.
\end{example}

The ranges in Example \ref{Ex2b} can be extended  to resemble Example \ref{Ex2}(c)  if  we just want there to be at least one residue class  whose image does not hit all classes.

\begin{example} \label{Ex2c}  Suppose that $p\equiv 1$ mod $4$ and  $f(x)=Ax^{(p+1)/2}$ mod $p$  and $2^{\beta}\parallel n$.

\vspace{1ex}
\noindent
 If $A$ satisfies \eqref{range} and
\begin{itemize}
\item[(a)]  $2^{\beta}\mid A$ and  $|A|<n$ ,  or
\item[(b)]  $2^{\beta}\nmid A$ and $|A|+\gcd(n,A) <n$,  or
\end{itemize}
$A$  is of the form \eqref{fracA}, and
\begin{itemize}
\item[(c)] $n$ is odd, with $|r|+|s|+\min\{ \gcd(n,r),\gcd(n,s)\}-1 \leq n$, or
\item[(d)]  $n$ is even and $2^{\beta}\mid r,$ with   $|r|+|s|+ \gcd(n,s) -1\leq n$,  or
\item[(e)]  $n$ is even and $2^{\beta}\mid s,$ with   $|r|+|s|+ \gcd(n,r) -1 \leq n$, or
%and $J:\equiv 2^{-1} p$ \; mod $\frac{n}{\gcd(n,r)}$, or
\item[(f)]  $n$ is even and $2^{\beta}\nmid rs$ with   $|r|+|s|+ \gcd(n,s) +\gcd(n,r) -1\leq n$,
\end{itemize}
%$$J:\equiv \frac{1}{2} \left(\frac{r}{\gcd(r,n)}\pm 1\right)\left(\frac{r}{\gcd(r,n)}\right)^{-1}\hspace{-2ex}p \;\; \text{ mod } %\frac{n}{\gcd(n,2r)}, $$
then  $f(I_i)\cap I_j=\emptyset$ for some $i,j$.

%If the restriction takes the form $B<n$ then in each case the number of missed residue classes $I_j$  will be at least $n-B$.

\end{example}
Appropriate values for   $i$ can be found in the proof  of Example \ref{Ex2c}, and again, for those $i$ there will be at least  $(n-B)$ missed residue classes $I_j$, when the restriction takes the form $n<B$ (although in some cases of (c) we must interchange the roles of $i$ and $j$).

 It turns out  that, as long as we avoid exponents $k=1$ or $(p+1)/2$ with  coefficients  similar to those  in Examples \ref{Ex2}, \ref{Ex2b} or \ref{Ex2c} then
  $f(I_i)$ will hit all residue classes once $p$ is sufficiently large relative to $n$. To make this precise we define the set
$$ \mathscr{C}: = \{C\equiv A x^{k-1} \text{ mod } p \; :\; 1\leq x\leq p-1,\;\; |C|<p/2\}. $$
Notice that for any integer $x$,  $f(x)=Ax^k\equiv Cx$ mod $p$ for some $C$ in $\mathscr{C}$. As we
shall see in Section \ref{dlarge(iv)}, when $d=\gcd(k-1,p-1)$ is relatively
large, and so
$|\mathscr{C}|$ is relatively small,  it can be useful to reduce to the consideration of the linear maps $Cx$ mod $p$.
Note that when $k=1$, $\mathscr{C}=\{A\}$, while when $k=\frac {p+1}2$, $\mathscr{C}=\{A,-A\}$.  In the next theorem we show that if $\mathscr{C}$ contains an element $C$ with $n \le |C| \le p/n$, and $p$ is sufficiently large  then $f(I_i)$ will hit all residue classes $I_j$. In particular, this happens when $A$ itself satisfies $n \le |A| \le p/n$. This is  always the case when $n=2$, other than the maps $f(x)=\pm x$ or the $\pm x^{(p+1)/2}$.

 If $\mathscr{C}$ contains only elements in the ranges $|C|<n$ or  $p/n <|C |< p/2$ then, prompted by the examples in Example \ref{Ex2}, \ref{Ex2b} and \ref{Ex2c}, we write the latter $C$ in the form
\be \label{defC}   C=\frac{tp-r}{s}, \;\;s>0, \;\;\gcd(r,s)=1. \ee
If for some such $C$, $|r|$ is sufficiently large relative to $s$ then again we see that the image of each  residue class will hit every residue class. Throughout the paper $x^{-1}$ mod $m$ denotes the multiplicative inverse of $x$ mod $m$.

\begin{theorem} \label{conj} If $\mathscr{C}$ contains an element $C$  or $C^{-1}$ mod $p$ with $n\leq |C|\leq p/n$  or
$$      C =\frac{tp-r}{s}, \;\;s>0, \;\; \gcd(r,s)=1,\, \;(n+3)s \leq |r|\leq \frac{ p}{n},   $$
and
$p\geq 4 \cdot 10^{29}\: n^{184/3},$
then  $f(I_i)\cap I_j\neq \emptyset$  for all $i,j$.

\end{theorem}

We show in Section \ref{proofmain} that any $C$ can be written in the form \eqref{defC} with
\be \label{boxP}  | r|<p/n, \;\; 1\leq s\leq n. \ee
Plainly once $r,s$ are chosen there will be only one value of $t\equiv rp^{-1}$ mod $s$ making $C$ an integer with $|C|<p/2$.
%\be \label{boxP}  1\leq t\leq  \left\lceil  n/2\right\rceil,\;\;\;| r|<p/n, \;\; 2t\leq s\leq n, \ee
%where $s$ is the nearest integer to $tp/C$.
In particular, for fixed $n$ there will be at most $(n+2)^3$ values of  $C$
which cannot be used in  Theorem \ref{conj}. Thus  if $|\mathscr C|>(n+2)^3$ we are guaranteed a  suitable $C$.
It turns out that we just need $|\mathscr{C}|>2$:

\begin{theorem} \label{mainiv} Suppose that $f(x) \neq Ax$ or $Ax^{\frac {p+1}2}$ mod $p$,
and that
$$
p > 4\cdot 10^{29} \: n^{184/3}.
$$
Then for any
$i,j$ we have $f(I_i) \cap I_j \neq \emptyset$.
\end{theorem}

For the exponents $k=1$ or $(p+1)/2$,  success or failure  depends critically  on the representation of $A$ in the manner \eqref{fracA} as we saw in Examples \ref{Ex2}, \ref{Ex2b} and \ref{Ex2c}. 
% with the restriction \eqref{sharp} in Example \ref{Ex2} sharp. 
In the linear case  we obtain a precise description of
the Type  (iv)  maps. 
 The restriction \eqref{sharp} in Example \ref{Ex2} is in fact sharp for $p$ sufficiently large.

\begin{theorem} \label{mainivk=1} Suppose that $f(x) = Ax$  mod $p$.

\noindent
 If  $p>n^3(n+3)$ then $f(x)$ is a Type (iv)  map if and only if $A$ is of the form
\be \label{critical}   A=\frac{tp-r}{s},\;\;\; \gcd(r,s)=1,\;\;\; s>0\;\;\; 1\leq  |r|+s \leq n. \ee

\end{theorem}

Writing
$$   S(N):= \sum_{\substack{1\leq r,s\leq N, \\ (r,s)=1,\\ r+s\leq N}} 1  \sim \frac{3}{\pi^2} N^2,$$
we see that for each $n\geq 3$ and $p$ we have  precisely  $2S(n)\sim (6/\pi^2) n^2$, choices of $A\equiv rs^{-1}$ mod $p$ that can give a Type  (iv) map $Ax$ mod $p$.

We obtain the same restriction  \eqref{sharp} for $k=(p+1)/2$ when $p$ is slightly larger.

\begin{theorem} \label{mainivk=(p+1)/2} Suppose that  $p\equiv 1$ mod $4$ and $f(x)=Ax^{\frac {p+1}2}$ mod $p$ with
$$p>\max\{(n^3+1)^2,8\cdot 10^4   (n\log n)^4\}.$$
If $A$ is not of the form \eqref{critical}
then for any
$i,j$ we have $f(I_i) \cap I_j \neq \emptyset$.
\end{theorem}

Example \ref{Ex2c} shows that for $k=(p+1)/2$ there are cases where  the condition $|r|+s\leq n$ is sharp, for example when $\gcd(n,r)$ or $\gcd(n,s)=1$. This time not every $r,s$ satisfying \eqref{sharp} will produce a Type (iv) map, but from Example \ref{Ex2c} we will get at least $2S((n+1)/2)\sim (3/2\pi^2) n^2$ and at most $2S(n)\sim (6/\pi^2) n^2$  examples of Type (iv)  maps $f(x)=Ax^{(p+1)/2}$ mod $p$.

%For a given $n$ we know that there are at most finitely many occurrences of Type (iv),  but of course the bounds in %this paper are far too large to
%obtain a complete determination as was done for $n=2$ in \cite{CochKony}. We hope to employ the methods of %\cite{CochKony} to complete this determination in a subsequent work.

\section{Computations and Conjectures}

Computations looking for maps of Type (iv) were performed for the primes $p<20,000$ and  moduli $n=3$ through 12.

These computations  revealed a number of families of Type (iv) maps that seemed to occur for every prime. These
all had exponent $k=1$ or $k=(p+1)/2$. Restricting to $k\neq 1$ or $(p+1)/2$, examples of Type (iv) eventually died out.
We showed in Theorem \ref{mainiv} that for a given $n$ there is indeed a  $C(n)$ such that once $p>C(n)$ any $f(x)=Ax^k$ mod $p$ with $k\neq 1, (p+1)/2$ has $f(I_i)\cap I_j\neq \emptyset$ for all $i,j$. The value $C(n)=4\cdot 10^{29} n^{184/3}$  obtained there is likely far from optimal.
For each $n=3$ through 12 the five  largest  primes  $p<20,000$ having  an  $f(x)=Ax^k$ mod $p$  with $k\neq 1$, $(p+1)/2$  and  $f(I_i)\cap I_j= \emptyset$  for some  $(i,j)$  are recorded in
Table \ref{notk=1}. Notice that   if $Ax^k$ has this property with  $2j\equiv p $ mod $n$  then so will $Ax^{k'}$ when $k'=k\pm (p-1)/2$ has $(k',p-1)=1$; a number of these pairs can be seen in the table.

\begin{table}[tbp]

\parbox{.4\linewidth}{
\centering

\small\begin{tabular}{|ccccc|}\hline
 & $p$ & $A$ & $ k$ & $(i,j)$ \\\hline
 & 83 & 21,26 & 81 & (1,1) \\
  & 89 & 17,21  & 23,67 & (1,1) \\
$n=3$ & 97 & 17 & 47,95 & (2,2) \\
  & 109 & 44 & 53,107 & (2,2)\\
 &  127 & 45,53 & 71 & (2,2) \\ \hline
  & 151 & 2 & 13 & (1,4),(2,3) \\
  & 151 & 46 & 127 & (3,1),(4,2) \\
  & 157 & 64 & 155 & (2,2),(3,3)\\
$n=4$ & 167 & 83 & 165 & (1,1),(2,2)\\
 & 193  & 16,48 & 95 & (2,2),(3,3) \\
 &193 & 49 & 95 & (2,3),(3,2)\\
 & 271  & 107 & 269 & (1,1),(2,2) \\\hline
 & 479 & 142 & 477 &  (2,2)\\
 & 503  & 25 &   65 &  (4,4) \\
$n=5$   & 503 &  243  & 363 & (4,4) \\
  & 521  &  215  & 259,519  & (3,3) \\
  &  541  & 176 &  269,539 & (3,3) \\
  & 601  &  59  & 251,551 & (3,3)  \\\hline
  & 449  & 158 & 447  & (5,5),(6,6) \\
  & 457  & 137 & 151 & (3,3),(4,4) \\
  & 457 & 162 & 227 & (1,1),(6,6)\\
$n=6$ & 457 & 80,137 & 455 & (3,3),(4,4)\\
 & 479  & 214 & 477 & (5,5),(6,6) \\
 & 547 & 30 & 155 & (3,3),(4,4)\\
 & 571  & 118 & 341 & (3,3),(4,4) \\\hline
 & 1303 &  347  & 1301  &  (4,4) \\
$n=7$  & 1321 &  232 &  329,989 &  (6,6)\\
 & 1409 & 416  &  703,1407 & (1,1) \\
 & 1489  & 653 &   371,1115 & (6,6)\\
 & 1733 &  670 &  865,1731 &  (2,2) \\ \hline
\end{tabular}

}
\hfill
\parbox{.4\linewidth}{
\centering

\small\begin{tabular}{|ccccc|}\hline
 & $p$ & $A$ & $ k$ & $(i,j)$ \\\hline
& 1249 &36 &623 & (1,1),(8,8)\\
 & 1301 & 432 &  599 &  (5,5),(8,8) \\
$n=8$ & 1381  & 648  & 1379 &  (5,8),(8,5)\\
 & 1637 & 437 & 1635 &  (6,7),(7,6)\\
 & 1777 &176 & 1775 &  (3,6),(6,3) \\ \hline
 & 2857 & 1383  & 713,2141  & (2,2) \\
 &  3037 & 105  & 505,2023 & (2,2) \\
$n=9$  &  3067 & 356 & 1871 & (8,8)\\
 & 3067 &  1313&  2363 & (8,8)  \\
 & 3089  & 482  & 1543,3087 & (1,1)\\
 & 3433 &  1590 &  571,2287 & (2,2) \\\hline
$n=10$ &
2137 &  830 & 1067 & (8,9),(9,8)\\
& 2287  & 109  & 2285 & (1,1),(6,6)\\
& 2377  &  623  & 2375  & (0,7),(7,0)\\
 & 2441  & 1169 &  1829  & (0,0),(1,1)\\
& 2473  & 803  & 1235  & (0,3),(3,0)\\ \hline
$n=11$
& 4787 & 624 & 4785 & (1,1) \\
& 4987 &2070 & 2215 & (2,2)\\
& 5281  & 964  & 2111,4751  & (6,6) \\
& 5683  & 2390 & 5681  & (9,9)\\
& 6577  & 731,3284 & 1645,4933 & (5,5)\\\hline
$n=12$
& 3457 & 1135 & 1727 & (0,1),(1,0)\\
& 3529 & 1485 &   1763 & (0,1),(1,0)\\
& 3637 & 993 & 3635  & (0,1),(1,0) \\
& 3659 & 934  & 3657  & (0,0),(11,11)\\
& 3851  &  9  & 351 & (5,6),(6,5) \\\hline
\end{tabular}
\vspace{11ex}
}

\vspace{1ex}
\caption{Type  (iv): Five largest $p<20,000$ with an  $f(x)=Ax^k$ mod $p$, $k\neq 1$, $(p+1)/2$ having $f(I_i)\cap I_j=\emptyset$ for some $(i,j)$. }
\label{notk=1}

\end{table}
In view of this data it is tempting to make the following conjecture.

\begin{conjecture} \label{Conj1}
 For $n=3$ through $12$ the optimal $C(n)$ is
\begin{align*} &  C(3)=127,\;\; C(4)=271,\;\; C(5)=601,\;\; C(6)=571, \;\; C(7)=1733,\\
 & C(8)=1777,\;\;C(9)=3433,\;\; C(10)=2473, \;\; C(11)=6577,\;\;C(12)= 3851. \end{align*}
The data suggests that one can take  $C(n)= 6n^{3}$.

\end{conjecture}

It is noticeable that maps of the form $f(x)=Ax^{p-2}=Ax^{-1}$ mod $p$ appear frequently in the data; this is somewhat surprising since from Theorem \ref{lehmer} we know that there are no Type (iv) maps of this form for $p>65n^4,$ a much smaller bound than we have
for the general $k$. But we note that this map is a self inverse, and most of the remaining examples of Type (iv) maps in our table are also  self inverse maps.

The recurring Type (iv) maps  $Ax$ or $Ax^{(p+1)/2}$ all seemed to have  $A$ small or of the form \eqref{defC} with $r$ and $s$ small. Identifying  and explaining  these led  to Examples \ref{Ex2}, \ref{Ex2b} and \ref{Ex2c}. In practice these Examples went through many refinements
as additional data revealed  new forms.  We know from Theorem \ref{mainivk=1} that Example \ref{Ex2}(c) is sharp. The current version of Example \ref{Ex2c}  is  able to predict all the  repeat
Type (iv) maps  that we see  in our data for $n=3$ through 12
(though higher $n$ would probably lead to new refinements). Some further  fine tuning is certainly possible, for example if $r$ and $\lfloor r/\gcd(s,n)\rfloor$ or $s$ and $\lfloor s/\gcd(r,n)\rfloor$ have opposite parity then we just need $r+s\leq n$ in Example \ref{Ex2c}(c)  (see the proof of Example \ref{Ex2c} for this and other cases where the gcd term can be dropped). Computations
for $n=15$, $A=(p-9)/5$ produced no  Type (iv) maps  between 1489 and 2000, showing that (c) can not always be weakened to  $r+s\leq n$.
Our existing data already showed that the gcds can not be dropped in (b),(d),(e) and (f); for example $n=12$, $A=9$, $(p\pm 8)/3$, $(p\pm 3)/8$, $(p\pm 1)/6$ or $(p\pm 9)/2$. In order to see that both gcds were needed in (f) computations were
carried out on $n=24$, $A=(p-4)/15$ and Type (iv) did not always occur.

Example \ref{Ex2}(c) gives Type (iv) maps of the form  $f(x)=Ax$ mod $p$ that will occur for every $p$ (whenever $p$ is in the correct  congruence class to make that $A$ an integer).
These $A$ for $n=3$ to 12 are shown in Table \ref{Exk=1}.

Similarly when $p\equiv 1$ mod 4 and $k=(p+1)/2$,  Example \ref{Ex2c} gives us cases of Type (iv) maps  $f(x)=Ax^{(p+1)/2}$ mod $p$ that will occur for all $p$.  These $A$ for $n=3$ to 12 are shown in Table \ref{Ex(p+1)/2}.

After excluding the values of  $A$ in Tables \ref{Exk=1} and \ref{Ex(p+1)/2}, few additional Type (iv) exceptions were found
in a search of $p<20,000$ and $k=1$ or $(p+1)/2$; the largest prime  for each $n$ is  shown in Tables \ref{k=1} and \ref{k=(p+1)/2}.

\begin{conjecture} \label{Conj2}  Suppose that $f(x)=Ax$ or $Ax^{(p+1)/2}$ mod $p$ where $A$ satisfies \eqref{range} but is not of the form
$$ |A|<n \;\;\; \hbox{ or } \;\;\;  A=(tp-r)/s \;\;\hbox{ with }\;\; |r|+|s|\leq n,\;\; \gcd(r,s)=1, $$
then $f(I_i)\cap I_j\neq \emptyset$ for all $i,j$
once $p>c(n)$, with the data suggesting that one can take $c(n)=3n^3$. For small $n$ the optimal values are
\begin{align*}  & c(3)=17,\;\; c(4)=61,\;\;c(5)=137,\;\; c(6)=197,\;\; c(7)=277,\\
  & c(8)=937,\;\;   c(9)=653,\;\; c(10)=2297,\;\; c(11)=1061,\;\; c(12)=2857.\end{align*}
\end{conjecture}

By Theorems \ref{mainivk=1} and  \ref{mainivk=(p+1)/2}  this holds with $c(n)=O(n^4)$ for $k=1$, and  $c(n)= O(n^6)$ for $k=(p+1)/2$.

\begin{table}[tbp]

\small\begin{tabular}{|cl|}\hline
  $n$ & \;\;\;\;\;\;\;  $A$  \\\hline & \\
3 &  $1,2,  (p-1)/2.$ \\ & \\
4 & $1,2,3, (p-1)/2,(p\pm 1)/3$.  \\  & \\
5 & $1,2,3,4,  (p-1)/2,(p-3)/2,(p\pm 1)/3,(p\pm 2)/3,(p\pm 1)/4.$ \\  & \\
6   &  $1,2,3,4,5, (p-1)/2,(p-3)/2,(p\pm 1)/3,(p\pm 2)/3,(p\pm 1)/4,(p\pm 1)/5,
(2p\pm 1)/5.$ \\  & \\
7  &  $1,2,3,4,5,6, (p-1)/2,(p-3)/2,(p-5)/2,(p\pm 1)/3,(p\pm 2)/3,
(p\pm 4)/3, (p\pm 1)/4,$\\ & $(p\pm 3)/4, (p\pm 1)/5,(p\pm 2)/5, (2p\pm 1)/5,2(p\pm 1)/5,(p\pm 1)/6. $\\  & \\
8  &  $1,2,3,4,5,6,7, (p-1)/2,(p-3)/2,(p-5)/2,(p\pm 1)/3,(p\pm 2)/3,
(p\pm 4)/3,(p\pm 5)/3, $ \\ & $(p\pm 1)/4,(p\pm 3)/4, (p\pm 1)/5,(p\pm 2)/5, (p\pm 3)/5, (2p\pm 1)/5,2(p\pm 1)/5,(2p\pm 3)/5, (p\pm 1)/6,$\\ & $ (p\pm 1)/7,(2p\pm 1)/7,(3p\pm 1)/7. $ \\   &  \\
9 & $1,2,3,4,5,6,7,8,(p-1)/2,(p-3)/2,(p-5)/2,(p-7)/2,(p\pm 1)/3,(p\pm 2)/3,(p\pm 4)/3,(p\pm 5)/3,$\\
 & $ (p\pm 1)/4,(p\pm 3)/4,(p\pm 5)/4,(p\pm 1)/5,(2p\pm 1)/5,(p\pm 2)/5,(2p\pm 2)/5,(p\pm 3)/5,(2p\pm 3)/5,$ \\ & $(2p\pm 4)/5,(2p\pm 4)/5, (p\pm 1)/6,(p\pm 1)/7,(2p\pm 1)/7,(3p\pm 1)/7,(p\pm 2)/7,$\\ & $(2p\pm 2)/7,(3p\pm 2)/7,(p\pm 1)/8,(3p\pm 1)/8.$
\\ & \\
10 &     $1,2,3,4,5,6,7,8,9,(p-1)/2,(p-3)/2,(p-5)/2,(p-7)/2, (p\pm 1)/3,(p\pm 2)/3,(p\pm 4)/3,$\\
 & $(p\pm 5)/3,(p\pm 7)/3,(p\pm 1)/4,(p\pm 3)/4,(p\pm 5)/4,(p\pm 1)/5,(2p\pm 1)/5,(p\pm 2)/5,(2p\pm 2)/5,$\\ & $(p\pm 3)/5,(2p\pm 3)/5,(p\pm 4)/5,(2p\pm4)/5,(p\pm 1)/6,(p\pm 1)/7,(2p\pm 1)/7,(3p\pm 1)/7,$ \\ & $(p\pm 2)/7,(2p\pm 2)/7,(3p\pm 2)/7,(p\pm 3)/7,(2p\pm 3)/7,(3p\pm 3)/7,(p\pm 1)/8,(3p\pm 1)/8,$\\
 & $(p\pm 1)/9,(2p\pm 1)/9,(4p\pm 1)/9.$         \\
 & \\
 11 & $1,2,3,4,5,6,7,8,9,10,(p-1)/2,(p-3)/2,(p-5)/2,(p-7)/2,(p-9)/2,(p\pm 1)/3,(p\pm 2)/3,$ \\
 & $((p\pm 4)/3,(p\pm 5)/3,(p\pm 7)/3,(p\pm 8)/3,(p\pm 1)/4,(p\pm 3)/4,(p\pm 5)/4,(p\pm 7)/4, (p\pm 1)/5,$\\
 & $(2p\pm 1)/5,(p\pm 3)/5,(2p\pm 2)/5,(p\pm 3)/5,(2p\pm 3)/5,(p\pm 4)/5,(2p\pm 4)/5,(p\pm 6)/5,(2p\pm 6)/5,$ \\
 & $(p\pm 1)/6,(p\pm 5)/6,(p\pm 1)/7,(2p\pm 1)/7,(3p\pm 1)/7,(p\pm 2)/7,(2p\pm 2)/7,(3p\pm 2)/7,(p\pm 3)/7,$ \\
 & $(2p\pm 3)/7,(3p\pm 3)/7,(p\pm 4)/7,(2p\pm 4)/7,(3p\pm 4)/7,(p\pm 1)/8,(3p\pm 1)/8,(p\pm 3)/8,(3p\pm 3)/8,$\\
 & $(p\pm 1)/9,(2p\pm 1)/9,(4p\pm 1)/9,(p\pm 2)/9,(2p\pm 2)/9,(4p\pm 2)/9,(p\pm 1)/10,(3p\pm 1)/10. $ \\
  & \\
 12 & $1,2,3,4,5,6,7,8,9,10,11,(p-1)/2,(p-3)/2,(p-5)/2,(p-7)/2,(p-9)/2,(p\pm 1)/3,(p\pm 2)/3,$\\
  & $(p\pm 4)/3,(p\pm 5)/3,(p\pm 7)/3, (p\pm8)/3,(p\pm 1)/4,(p\pm 3)/4,(p\pm 5)/4, (p\pm7)/4,(p\pm 1)/5,$\\ & $(2p\pm 1)/5,(p\pm 2)/5,(2p\pm 2)/5,(p\pm 3)/5,(2p\pm 3)/5,(p\pm 4)/5,(2p\pm 4)/5,$\\ & $(p\pm 6)/5,(2p\pm 6)/5,(p\pm 7)/5,(2p\pm 7)/5,(p\pm 1)/6, (p\pm5)/6,(p\pm 1)/7,(2p\pm 1)/7,(3p\pm 1)/7,$  \\ & $(p\pm 2)/7,(2p\pm 2)/7,(3p\pm 2)/7, (p\pm 3)/7,(2p\pm 3)/7,(3p\pm 3)/7,(p\pm 4)/7,(2p\pm 4)/7,(3p\pm 4)/7,$\\ & $(p\pm 5)/7,(2p\pm 5)/7,(3p\pm 5)/7,(p\pm 1)/8,(3p\pm 1)/8, (p\pm3)/8, (3p\pm3)/8,(p\pm 1)/9,(2p\pm 1)/9,$\\ & $(4p\pm 1)/9, (p\pm2)/9, (2p\pm2)/9, (4p\pm2)/9,(p\pm 1)/10,(3p\pm 1)/10, (p\pm 1)/11,(2p\pm 1)/11,$\\ & $(3p\pm 1)/11,(4p\pm 1)/11,(5p\pm 1)/11.$ \\
& \\ \hline
\end{tabular}

\vspace{1ex}
\caption{Type (iv) examples  $Ax$ mod $p$   from Example \ref{Ex2}.}
\label{Exk=1}
\end{table}

\vspace{3ex}
\begin{table}
\small\begin{tabular}{|cl|}\hline
  $n$ & \;\;\;\;\;\;\;  $A$  \\\hline & \\
3 &  $1,2,  (p-1)/2.$ \\ & \\
4 & $1.$  \\  & \\
5 & $1,2,3,4,  (p-1)/2,(p-3)/2,(p\pm 1)/3,(p\pm 2)/3,(p- 1)/4.$ \\  & \\
6   &  $1,2,4, (p-1)/2,(p-1)/4.$ \\  & \\
7  &  $1,2,3,4,5,6, (p-1)/2,(p-3)/2,(p-5)/2,(p\pm 1)/3,(p\pm 2)/3,
(p\pm 4)/3,$\\   & $ (p- 1)/4,(p+ 3)/4, (p\pm 1)/5,(2p\pm 1)/5, (p\pm 2)/5,(2p\pm 2)/5,(p\pm 1)/6. $\\  & \\
8  &  $1,2,3,5, (p-1)/2,(p-3)/2,(p\pm 1)/3,(p\pm 2)/3, (p\pm 1)/5, (2p\pm 1)/5. $\\  & \\
9 & $1,2,3,4,5,6,7,8,(p-1)/2,(p-3)/2,(p-5)/2,(p-7)/2,(p\pm 1)/3,(p\pm 2)/3,$\\
& $(p\pm 4)/3,(p\pm 5)/3,
  (p- 1)/4,(p+3)/4,(p- 5)/4,(p\pm 1)/5,(2p\pm 1)/5,(p\pm 2)/5,$ \\
& $(2p\pm 2)/5,(p\pm 3)/5,(2p\pm 3)/5,(p\pm 4)/5,(2p\pm 4)/5, (p\pm 1)/6,(p\pm 1)/7,$\\
& $(2p\pm 1)/7,(3p\pm 1)/7,(p\pm 2)/7,(2p\pm 2)/7,(3p\pm 2)/7,(p- 1)/8,(3p+ 1)/8.$
\\ & \\
10 & $1,2,3,4,6,7,8,(p-1)/2,(p-3)/2, (p-7)/2,(p\pm 1)/3,(p\pm 2)/3,(p\pm 4)/3,(p\pm 1)/4,(p\pm 3)/4,$\\
& $(p\pm 1)/6,(p\pm 1)/7,(2p\pm 1)/7,(3p\pm 1)/7,(p\pm 2)/7,  (2p\pm 2)/7,(3p\pm 2)/7, (p-1)/8, (3p+1)/8.$\\
& \\
 11 & $1,2,3,4,5,6,7,8,9,10,(p-1)/2,(p-3)/2,(p-5)/2,(p-7)/2,(p-9)/2,(p\pm 1)/3,(p\pm 2)/3,$ \\
 & $((p\pm 4)/3,(p\pm 5)/3,(p\pm 7)/3,(p\pm 8)/3,(p- 1)/4,(p+ 3)/4,(p- 5)/4,(p+ 7)/4, (p\pm 1)/5,$\\
 & $(2p\pm 1)/5,(p\pm 3)/5,(2p\pm 2)/5,(p\pm 3)/5,(2p\pm 3)/5,(p\pm 4)/5,(2p\pm 4)/5,(p\pm 6)/5,(2p\pm 6)/5,$ \\
 & $(p\pm 1)/6,(p\pm 5)/6,(p\pm 1)/7,(2p\pm 1)/7,(3p\pm 1)/7,(p\pm 2)/7,(2p\pm 2)/7,(3p\pm 2)/7,(p\pm 3)/7,$ \\
 & $(2p\pm 3)/7,(3p\pm 3)/7,(p\pm 4)/7,(2p\pm 4)/7,(3p\pm 4)/7,(p- 1)/8,(3p+1)/8,(p+ 3)/8,(3p- 3)/8,$\\
 & $(p\pm 1)/9,(2p\pm 1)/9,(4p\pm 1)/9,(p\pm 2)/9,(2p\pm 2)/9,(4p\pm 2)/9,(p\pm 1)/10,(3p\pm 1)/10. $ \\
  & \\
12 & $1,2,3,4,5,7,8,(p-1)/2,(p-3)/2,(p-5)/2,(p-7)/2,(p\pm 1)/3,(p\pm 2)/3,(p\pm 4)/3,(p\pm 5)/3,$\\
 & $ (p- 1)/4,(p+3)/4, (p-5)/4, (p+7)/4,(p\pm 1)/5,(2p\pm 1)/5,(p\pm 2)/5,( 2p\pm 2)/5,(p\pm 3)/5,$\\
 & $(2p\pm 3)/5,(p\pm 4)/5,(2p\pm 4)/5,(p\pm 1)/7,(2p\pm 1)/7,(3p\pm 1)/7,(p\pm 2)/7,(2p\pm 2)/7 ,(3p\pm 2)/7,$\\
 & $(p\pm 4)/7,(2p\pm 4)/7,(3p\pm 4)/7,(p-1)/8, (3p+1)/8.$ \\
 & \\\hline
\end{tabular}

\vspace{1ex}
\caption{Type (iv) examples  $Ax^{(p+1)/2}$ mod $p$   from Example \ref{Ex2c}.}
\label{Ex(p+1)/2}

\end{table}

\begin{table}[tbp]

\small\begin{tabular}{|ccccc|}\hline
 & $p$ & $A$ & $ k$ & $(i,j)$ \\\hline
$n=3$  & 13  & 5 &    1 & (2,2)\\\hline
$n=4$  &  19 &  7 & 1 &  (3,4),(4,3))\\
  & 19  & 8 & 1  & (3,3),(4,4) \\ \hline
$n=5$  & 53  &  14,19  & 1  & (4,4)  \\\hline
 $n=6$  & 61  & 16,22  & 1  & (2,4),(5,3) \\
   & 61 &   19 &  1 & (3,2),(4,5) \\
  & 61  &  25  &  1 & (3,5),(4,2)) \\\hline
$n=7$   & 131  &  27,34 &  1  & (6,6)\\  \hline
$n=8$ & 151 & 31,39 & 1 & (7,8),(8,7) \\ \hline
$n=9$ &  241 & 35,62 & 1 & (8,8) \\ \hline
$n=10$ &
283 & 48,58  & 1  & (4,8),(9,5)\\
& 283 &  112  &  1  & (5,4),(8,9)\\
& 283  & 122  & 1 & (5,9),(8,4) \\ \hline
$n=11$
& 449 &  65,76 & 1 & (10,10) \\ \hline
$n=12$
& 491  & 71,83 &  1 & (0,11),(11,0)\\ \hline
\end{tabular}

\vspace{1ex}
\caption{Type (iv):  Largest $p<20,000$ having  an $f(x)=Ax$ mod $p$ with $f(I_i)\cap I_j=\emptyset$ for some $(i,j)$ and  $A$  not in Table \ref{Exk=1}.}
\label{k=1}

\end{table}

\begin{table}[tbp]

\small\begin{tabular}{|ccccc|}\hline
 & $p$ & $A$ & $ k$ & $(i,j)$ \\\hline
$n=3$   & 17 &  5  & 9  &   (2,2),(3,3)\\
  & 17 & 7 & 9 & (2,3),(3,2) \\ \hline
$n=4$  & 61  & 6 & 31  & (1,3),(4,2)\\
 & 61 & 10 & 31 & (2,4),(3,1) \\\hline
$n=5$    & 137  &   7  &  69  &  (3,2),(4,5)\\
  & 137  &  39 &  69  & (2,4),(5,3) \\  \hline
$n=6$ & 197 & 16 & 99 &(1,3),(4,2)\\
 & 197 & 37 & 99 & (2,4),(3,1) \\ \hline
 $n=7$ & 277 &  9,56 & 139 & (5,4),(6,7) \\
  & 277 & 62  & 139 & (4,5),(7,6)\\
  & 277 & 67 & 139 & (5,7),(6,4)\\
  & 277 &  94,123 & 139 & (4,6),(7,5) \\ \hline
 $n=8$ & 937 & 188 & 469 & (2,7),(7,2) \\
 & 937 & 314 & 469 & (2,7),(7,7) \\\hline
 $n=9 $ & 653 & 149 & 327 & (1,1),(4,4) \\ \hline
$n=10$ &
2297 &  768,984 &  1149 & (3,4),(4,3) \\\hline
 $n=11$ &
1061 & 337 & 531  & (5,7),(11,9)\\
 & 1061  & 488 & 531  & (7,5),(9,11)\\\hline
$n=12$ & 2857 & 570,817 & 1429 & (4,4),(9,9)\\\hline

\end{tabular}

\vspace{1ex}
\caption{ Type (iv): Largest $p<20,000$ having  an $f(x)=Ax^{(p+1)/2}$ mod $p$ with  $f(I_i)\cap I_j=\emptyset$ for some $(i,j)$ and $A$   not in Table \ref{Ex(p+1)/2}.}
\label{k=(p+1)/2}

\end{table}

\section{Type $(iv)$  intersections for large $d$}\label{dlarge(iv)}

Theorem \ref{conj} is an immediate consequence of the next two theorems.
Recall that $\mathscr C$ is the set of absolute least residues
$$
\mathscr C:= \{Ax^{k-1}\text{ mod } p: 1 \le x \le p-1\}.
$$
If we have a suitable $C\in \mathscr C$ then we can show that each residue class gets mapped to all the residue classes.

\begin{theorem}\label{biggcd2} Suppose that $\mathscr{C}$ contains an integer $C$ with $n\leq |C|\leq p/n$.

\noindent If  $p> 10^6 $
and $d\geq 0.88 n^2 p^{1/2}\log^2 p,$ or if $k=1$, then $f(I_i)\cap I_j \neq \emptyset$ for all $i,j$.
\end{theorem}

\begin{proof} We proceed as in the proof of Theorem 4.1 of \cite{CJGK1},  but with $I_j$ in place of $I\setminus I_j$.
Since $C$ is in $\mathscr{C}$ we can write $ C\equiv  AB^{k-1} \text{ mod } p$ for some $B$. We let
$$   L:=(p-1)/d,\;\; $$
and
$$\mathscr{U}=\{ x\in I_i \: : \;\; Cx \text{ mod } p \in I_j, \;\; x\equiv Bz^L \text{ mod } p \text{ for some $z$}\}. $$
Notice that if $x$ is in $\mathscr{U}$ we have
$$ Ax^k\equiv Cx (B^{-1}x)^{k-1} \equiv Cx z^{L(k-1)}=Cx (z^{p-1})^{(k-1)/d} \equiv Cx \text{ mod } p. $$
So we  can  show that $f(I_i)\cap I_j\neq \emptyset$ by showing   $|\mathscr{U}|>0$.
 Writing $\mathscr{I}_{j}(x)$ for the characteristic function of $I_j$, and using the Dirichlet characters of order $L$ to pick out when $B^{-1}x$ is an $L$th power mod $p$, we have
$$  L|\mathscr{U}|=\sum_{\chi^L=\chi_0} \chi(B^{-1})S(\chi), \;\;\;\;\; S(\chi):= \sum_{x\in \mathbb Z_p}
\chi(x)\mathscr{I}_i(x)\mathscr{I}_j(Cx), $$
where $\chi_0$ denotes  the principal character. Hence
$$ L|\mathscr{U}|= M_{ij}+E, $$
where
$$ M_{ij}:=\sum_{x\in \mathbb Z_p^*}  \mathscr{I}_i(x)\mathscr{I}_j(Cx),\;\;\;\;\; E:= \sum_{\substack{\chi^L=\chi_0\\ \chi\neq \chi_0}} \chi(B^{-1})S(\chi).$$
Using the finite Fourier expansion $ \displaystyle \mathscr{I}_i(x)= \sum_{ y\in \mathbb Z_p} a_i(y) e_p(yx)$ we have
$$ S(\chi) =  \underset{(u,v)\neq (0,0)}{\sum_{u=0}^{p-1} \sum_{v=0}^{p-1}} a_i(u)a_j(v)  \sum_{x\in \mathbb Z_p} \chi(x) e_p(ux+vCx). $$
The classic Gauss sum bound, and  the \cite[Theorem 1]{Coch1}  bound  on $\displaystyle \sum_{u=0}^{p-1}|a_i(u)|$, give as in the proof of  \cite[Theorem 4.1]{CJGK1},
\begin{equation} \label{Lminus1}
 |E|<  0.22  (L-1)  p^{1/2}\log^2 p.
\end{equation}
We need a lower bound on $M_{ij}$. Suppose that we have $n\leq C\leq p/n$. If  $C<0$ we replace $C$ by $-C$ and $j$ by $\overline{j}=p-j$ mod $n$.  Since $0<Cx <Cp$ we have
\begin{align*} M_{ij} &  = \sum_{u=0}^{C-1} |\{ x\in I_i \; : \; up\leq Cx < (u+1)p,\;\; Cx-up \in I_j\}| \\
 =  & \sum_{\substack{u=0\\u\equiv K \text{ mod } n}}^{C-1} \left|\left\{   x\in I_i \;  : \; \frac{up}{C} \leq x < \frac{up}{C} +\frac{p}{C}     \right\}\right|,
\end{align*}
where  $K:=(Ci-j)p^{-1}$ mod $n$.

Note that for $p/2n < C \leq p/n$ we have
$$ \left\lfloor \frac{p}{nC}\right\rfloor =1 > \frac{p}{2nC}, $$
and for $C\leq p/2n$
$$\left\lfloor \frac{p}{nC}\right\rfloor > \frac{p}{nC}-1  \geq \frac{p}{2nC}. $$
Similarly, for $n\leq C <2n$ we  have
$$ \left\lfloor \frac{C}{n} \right\rfloor =1 > \frac{C}{2n}, $$
and for $C\geq 2n$
$$ \left\lfloor \frac{C}{n} \right\rfloor > \frac{C}{n} -1 \geq \frac{C}{2n}. $$
Hence, observing that a general interval of length $\ell$ or an interval of the form $[0,\ell-1]$,  will contain at least $\lfloor \ell/n \rfloor$ complete sets of
residues  mod $n$, we have
$$   \left|\left\{   x\in I_i \;  : \; \frac{up}{C} \leq x < \frac{up}{C} +\frac{p}{C}     \right\}\right|  \geq \left\lfloor \frac{p}{nC}\right\rfloor > \frac{p}{2nC}, $$
and
$$ \left| \left\{  0\leq u  \leq C-1\; : \; u\equiv K \text{ mod } n\right\}\right| \geq  \left\lfloor \frac{C}{n} \right\rfloor
> \frac{C}{2n}, $$
giving
$$ M_{ij} > \frac{C}{2n} \cdot  \frac{p}{2nC} = \frac{p}{4n^2}. $$
Hence, as long as we have
$$\frac{p}{4n^2} \geq 0.22\frac{p^{3/2}\log^2p}{d},$$
 we have $|\mathscr{U}|>0$ and $f(I_i)\cap I_j\neq \emptyset$.

If $k=1$ then $C=A$ and as shown above, $|f(I_i) \cap I_j|=M_{ij}>p/4n^2$ whenever $n\leq |A| \leq p/n$.
Note that $n \le |C| \le p/n$ implies $p>n^2$.
\end{proof}

\begin{theorem}\label{Cform} Suppose that $\mathscr{C}$ contains  a $C$ of the form
\begin{equation} \label{goodC}
     C \hbox{ or } C^{-1} =\frac{(tp-r)}{s}, \;\;s>0, \;\; \gcd(r,s)=1, \  \;(n+3)s \leq |r|\leq \frac{ p}{n}.
     \end{equation}
 If  $p>10^6 $
and $d\geq 1.32  n^2 p^{1/2}\log^2 p$, or if $k=1$, then $f(I_i)\cap I_j  \neq \emptyset$ for all $i,j$.
\end{theorem}

\begin{proof} We proceed as in Theorem \ref{biggcd2}.  If $C^{-1}$ is of the stated form we observe that
counting $x$ in $I_i$ with $Cx$ in $I_j$ is the same as counting $x$ in $I_j$ with $C^{-1}x$ in $I_i$ and reverse the roles of $i$ and $j$. So we suppose that $C$ is of the stated form and that   $r,s>0$. If $r<0$ we
can replace $C$ by $-C$ and $j$ by $\overline{j}=p-j$ mod $n$.

To estimate $M_{ij}$ we split  the $x$ into the different residue
classes $a$ mod $s$ and observe that for $x=a+sy$ we have
$$ Cx = x \left(\frac{tp-r}{s}\right) \equiv \frac{(tp-r)a}{s} -ry \text{ mod } p. $$
Hence, writing   $\frac{(tp-r)a}{s}\equiv \alpha(a)$ mod $p$ with $0\leq \alpha(a)<p$,  we have
$$M_{ij}=\sum_{a=0}^{s-1} \left| \left\{ 0\leq y \leq \frac{(p-1-a)}{s} \; :\; ys+a \in I_i,  \; \alpha(a)-ry \text{ mod } p \in I_j\right\}\right|. $$
If $b:=\gcd(n,s)=1$ then the condition $ys+a\in I_i$ reduces to the  mod $n$ congruence $y\equiv \lambda(a):= (i-a)s^{-1}$ mod $n$. If $b>1$
then we are reduced to the $s/b$ values
$$\mathscr{A}=\{ a \; : \; 0\leq a < s,\; a\equiv i \text{ mod } b\},$$
and the condition
$ys + a \in I_i$ becomes $ y\equiv \lambda (a):= (s/b)^{-1}(i-a)/b$ mod $n/b$, that is  $y\equiv \lambda_v(a)$ mod $n$,
$v=1,\ldots ,b$  with $\lambda_v(a)=\lambda (a)+vn/b$.

Now any $y$ with
$$-\left(\left\lfloor \frac{r(p-1-a)}{sp}\right\rfloor -1\right)  p\leq \alpha(a)-ry  <  0$$
will have $0<y \leq (p-1-a)/s$ and  hence
$$ M_{ij}\geq \sum_{a\in \mathscr{A}} \sum_{v=1}^b \sum_{u=1}^{\lfloor \frac{r(p-1-a)}{sp}\rfloor -1   } M_{ij}(a,v,u), $$
where
$$M_{ij}(a,v,u)= \abs{ \{ y\equiv \lambda_v(a) \text{ mod } n/b,\;\;     -u p \leq \alpha(a)-ry <-(u-1)p,\; \alpha(a)-ry \text{ mod }p \in I_j\}  }. $$
The condition $\alpha(a)-ry \text{ mod } p \in I_j$ becomes $\alpha(a)-ry+up \equiv j$ mod $n$ and $u\equiv \mu(a,v):=(j+r\lambda_v(a)-\alpha(a))p^{-1}$ mod $n$.

Hence
$$ M_{ij}\geq \sum_{a\in \mathscr{A}} \sum_{v=1}^b
\sum_{\stackrel{u=1}{ u\equiv \mu(a,v) \text{ mod } n}} ^{\lfloor \frac{r(p-1-a)}{sp}\rfloor -1   }
\left| \left\{ y\equiv \lambda_v(a)  \text{ mod } n, \:  \frac{(\alpha(a)+up)}{r}  - \frac{p}{r} <y  \leq    \frac{(\alpha(a)+up)}{r}\right\} \right|. $$
 When $n<p/r<2n$ we observe that we are guaranteed at least one element $ y\equiv \lambda_v(a)$  mod $n$ in the interval of length $p/r>n$. When $p/r\geq 2n$ we use that we have
at least $\lfloor p/rn\rfloor>p/rn-1$ elements satisfying the congruence. Hence
$$\left| \left\{ y\equiv \lambda_v(a)  \text{ mod } n, \:  \frac{(\alpha(a)+up)}{r}  - \frac{p}{r} <y  \leq    \frac{(\alpha(a)+up)}{r}\right\} \right| \geq \frac{p}{2rn}.$$
Similarly,  with $(n+3)s \leq r<p/n$,
$$ \left\lfloor \frac{r(p-1-a)}{sp}\right\rfloor - 1 \geq  \frac{r(p-s)}{sp} -2  \geq   \frac{r}{s}-3 \geq  n. $$
So we get at least one $u$ in the sum  satisfying  $u\equiv \mu(a,v)$ mod $n$ for $(n+3) \leq r/s < (2n+3)$ and $\lfloor (r/s-3) /n\rfloor > r/ns-3/n -1$ for $(2n+3)\leq r/s$ and
$$ \left| \left\{ 1\leq u \leq  \left\lfloor \frac{r(p-1-a)}{sp}\right\rfloor \; : \; u\equiv \mu(a,v) \text{ mod } n\right\}\right| \geq \frac{r}{s(2n+3)}. $$

Hence
\begin{equation} \label{MijLB}
 M_{ij} \geq \frac{s}{b}\cdot b \cdot \frac{r}{s(2n+3)} \cdot \frac{p}{2rn} = \frac{p}{2n(2n+3)},
 \end{equation}
and making this greater than $|E|< 0.22 (p/d) \sqrt{p} \log^2 p$  ensures that $\mathscr{U}\neq \emptyset.$
The $k=1$ case follows as before.
\end{proof}

\section{Proofs of Theorems  \ref{conj} and \ref{mainiv}} \label{proofmain}

\begin{proof}[Proof of Theorem \ref{conj}] Suppose that $p> 4\cdot 10^{29} \:n^{184/3}$. Certainly $p>10^6$. If $d\leq 0.006\: p^{89/92}$ then Theorem \ref{conj} follows from Theorem  \ref{smallgcd2}, while if $d\geq   1.32 \: n^2 p^{1/2}\log^2 p$
it follows from Theorems \ref{biggcd2} and \ref{Cform}.  Otherwise we have $$ 1.32 \:n^2 p^{1/2}\log^2 p>d> 0.006 \:p^{89/92},$$ and so
$p^{43/92}/\log^2 p <220 n^2.$ But this  does not occur for $p>4 \cdot 10^{29}\: n^{184/3}$.
\end{proof}

\begin{proof}[Proof of Theorem \ref{mainiv}] Suppose that $k\neq 1$ or $(p+1)/2$ and that
$$
p > \max\{ 4 \cdot 10^{29} n^{184/3}, 2 (n+3)^2n^4\}= 4 \cdot 10^{29} n^{184/3}.
$$
 Let $C$ be an integer with $|C|<p/2$. By the box principle the congruence $y \equiv Cx$ mod $p$ has a nonzero solution $x=s$, $y=r$ with $1\leq s \leq n$, $|r|<p/n$, and gcd$(r,s)=1$.
 Let $\mathscr G$ be the group of $d$-th powers mod $p$, and recall that
$$ \mathscr{C}: = \{A x^{k-1} \text{ mod } p \; :\; 1\leq x\leq p-1\}=\{Ax \text{ mod }p\:  : \: x\in \mathscr{G}\}, $$
reduced to values between $-p/2$ and $p/2$.
Each element  $C \in \mathscr C$ has a representation as above,
$$
C \equiv rs^{-1} \text{mod $p$}, \quad 1 \le s \le n, \quad |r|< p/n, \quad \text{gcd$(r,s)=1$}.
$$
 If for some $C \in \mathscr C$ we have $(n+3)s\le |r|\le p/n$, then Theorem \ref{conj} applies.

 Otherwise, every $C \in \mathscr C$ is in
$$
\mathscr{B}:= \{C \in \mathscr C: C \equiv rs^{-1} \text{ mod } p,\  1 \le s \le n,\  |r|<(n+3)s\}.
$$
In this case, let $Ax$ be an element in $\mathscr C$ having a representation $Ax \equiv rs^{-1}$ mod $p$ from $\mathscr{B}$ with $|r/s|$ minimal. Let $y \neq \pm 1 \in \mathscr G$; such a $y$ exists since $|\mathscr G| \ge 3$ by assumption.  Then $Ax$, $Ayx$, $Ay^{-1}x$ are distinct elements of $\mathscr C$, and so we  have representations
$$
Ax\equiv r_1s_1^{-1}, \quad \quad Ayx \equiv r_2s_2^{-1}, \quad \quad Ay^{-1}x \equiv r_3s_3^{-1} \quad \text { mod } p,
$$
with $1\leq s_i\leq n$, $|r_i|<(n+3)s_i$. Thus,
$$
y \equiv r_2s_2^{-1}s_1r_1^{-1} \equiv r_1s_1^{-1}s_3r_3^{-1} \text{ mod } p
$$
and so
$$
s_1^2r_2r_3 \equiv r_1^2 s_2s_3 \text { mod } p.
$$
Thus if $p>2 (n+3)^2n^4$ then the two sides must be equal, that is,
$$
\left(\frac {r_1}{s_1}\right)^2 = \frac {r_2}{s_2}\frac {r_3}{s_3},
$$
which cannot happen by the minimality of $r_1/s_1$.
\end{proof}

\section{Proofs of  Theorems \ref{mainivk=1} and  \ref{mainivk=(p+1)/2}} \label{mainivks}

 In order to deal with the exponents $k=1$ and $k=(p+1)/2$, we need the following addition to Theorems \ref{biggcd2} and \ref{Cform} which deals with the case when $r,s$ are both small but \eqref{sharp} does not hold.

\begin{theorem}\label{smallrs}
Suppose that
$$A=\frac{tp-r}{s},\;\;\; s>0,\;\;\; \gcd(r,s)=1,\;\; |r|+s >n.  $$
\noindent
(a) If $p>|r|sn,$ then $f(x)=Ax$ mod $p$ has $f(I_i)\cap I_j\neq \emptyset$ for all $i,j$.

\noindent
(b) If $p>(|r|sn+1)^2$  then $f(x)=Ax^{(p+1)/2}$ mod $p$ has $f(I_i)\cap I_j\neq \emptyset$ for all $i,j$.

\end{theorem}

\noindent
 For part (a) we actually prove that $|f(I_i) \cap I_j| \ge \lfloor p/rns\rfloor$ under the given hypotheses.
Using \cite{Burgess} we can replace the hypothesis in part (b) with the condition  $p\gg (|r|ns \log (|r|ns))^{4/3}.$

\begin{proof} (a) We first deal with the linear case $f(x)=Ax$ mod $p$. We assume that $r>0$ else we can replace $A$ by $-A$ and $j$ by $\bar{j}=p-j$ mod $n$.
We also assume that $r<s$ else we replace $A$ by $A^{-1}=(t'p-s)/r$ where $t'\equiv sp^{-1}$ mod $r$ and switch the roles of
$i$ and $j$.

Take $a$ with $1\leq a\leq n$ with
$$ a\equiv (js+ri) p^{-1} \text{ mod } n. $$
For convenience here $p^{-1}$ will denote the inverse of $p$ mod $ns$.

We define $u$ such that
$$ u\equiv i \text{ mod  } n,\;\;\; u\equiv at^{-1} \text{ mod } s. $$
Writing $b:=\gcd(n,s)$ we see that if $b>1$ then  $at^{-1}\equiv i rp^{-1}t^{-1} \equiv i$ mod $b$ so there is a solution (defined mod $ns/b$). Note that
$ap-js-ru \equiv 0$ mod $s$ and mod $n$, and so when $b>1$  we can define $\lambda$ by
$$ r\lambda\equiv  \frac{(ap-js-ru)}{(ns/b)}   \text{ mod } b, \;\; 0\leq \lambda <b, $$
with $\lambda=0$ if $b=1$. Set $v=u+\lambda ns/b$. We split into two cases:

\vspace{1ex}
\noindent
{\bf Case 1:}  $1\leq a\leq s$.

We solve
\begin{equation} \label{xcrt1}
x\equiv v \text{ mod } ns,\;\;\;  1\leq x\leq  \min \{ ap/r,p-1\}.
\end{equation}
The condition $p>nrs$ ensures that $ap/r\geq p/r>ns$  so we are guaranteed a solution, and $x\equiv i$ mod $n$ so $x$ is in $I_i$.

Since $xt\equiv a$ mod $s$ we have
$$ Ax\equiv \frac{ap-xr}{s}, \text{ mod } p. $$
Notice that $0<(ap-xr)/s< ap/s\leq p$  so that this is the least residue with
$$ \frac{ap-rx}{s} \equiv \frac{ap-rv}{s} =j + n\frac{ (ap-js-ru)/(ns/b)-r\lambda}{b} \equiv j  \text{ mod } n. $$

\vspace{1ex}
\noindent
{\bf Case 2:} $s+1\leq a\leq n$.

Notice $1\leq a-s \leq n-s<r$. We solve
\begin{equation} \label{xcrt2}
x\equiv v \text{ mod } ns,\;\;\;  (a-s)p/r < x\leq  p-1.
\end{equation}
Since $0<(a-s)p/r\leq p-p/r < p-ns$  we are again  guaranteed a solution $x$ in $I_i$,
$$ Ax\equiv \frac{(a-s)p-xr}{s} +p, \text{ mod } p. $$
Since $0>((a-s)p-xr)/s >-pr/s>-p$ this is the least residue and again $(ap-rx)/s\equiv j $ mod $n$.

We note that the set of $x$ satisfying \eqref{xcrt1} or \eqref{xcrt2}, is an arithmetic progression of length
at least $\lfloor p/rns \rfloor$. In particular, we have shown that
$M_{ij}=\{x\in I_i \; : Ax \text{ mod $p$ } \in I_j\}$ satisfies   $|M_{ij}|\geq \lfloor p/rsn\rfloor.$

(b) Suppose now that $k=(p+1)/2$, and that $p>(|r|sn +1)^2$.  Then $f(x) \equiv \pm Ax$ mod $p$ depending on whether $x$ is a quadratic residue or not.  In part (a) we saw that there was an arithmetic progression of $\lfloor p/|r|sn\rfloor \ge p/|r|sn -1 >\sqrt{p}$  values of $x\in I_i$, with $Ax \text{ mod $p$ } \in I_j$. By \cite{Hummel} these  cannot all be quadratic nonresidues. Thus we must  have a quadratic residue  $x \in I_i$ with
$f(x)=Ax$ mod $p$  in $I_j$.
 \end{proof}

\begin{proof}[Proof of Theorem \ref{mainivk=1}]  Suppose that $p>(n+3)n^3$ and that $f(x) = Ax$ mod $p$. By the box principle we can write  $A\equiv r s^{-1}$ mod $p$ with
$(r,s)=1$, $1\leq s\leq n$ and $|r|<p/n$.   If $1 \le |r|+s \le n$, then Example \ref{Ex2}(c) shows that $f(x)=Ax$ mod $p$ is a Type (iv) mapping. Suppose now that $|r|+s>n$.
 If $|r|>(n+3)s$ then the result follows from Theorem \ref{Cform},
so we can assume that $|r|\le (n+3)s\le (n+3)n.$ Since $p>(n+3)n^3$ we have $p>|r|sn,$ and so
Theorem \ref{smallrs} gives  $f(I_i)\cap I_j\neq \emptyset$ for all $i,j$.
\end{proof}

\begin{proof}[Proof of Theorem \ref{mainivk=(p+1)/2}] Suppose that $k=(p+1)/2$ and  $p>\max\{(n^3+1)^2,8\cdot 10^4   (n\log n)^4\}$.
Observe that in the proof of Theorem \ref{Cform} we have $L=2$
and hence by \eqref{Lminus1} will get $f(I_i)\cap I_j\neq \emptyset$ as long as $M_{ij}> 0.22 \sqrt{p} \log^2p$.
Since
$p> 8 \cdot 10^4  (n\log n)^4$ we have $p>8\cdot 10^6$ and so by \eqref{MijLB},
\be \label{bigrange}  M_{ij} \ge p/2n(2n+3)\geq p/6n^2 >0.22 \sqrt{p}\log^2 p, \ee
provided $\mathscr{C}$ contains a value $C$ satisfying \eqref{goodC}. 

By the box principle we can write  $A\equiv r_1 s_1^{-1}$ mod $p$ and $A^{-1}\equiv r_2 s_2^{-1}$  mod $p$ with
$(r_i,s_i)=1$, $1\leq s_i\leq n$ and $|r_i|<p/n$.
If  one of these has $|r_i|\geq (n+3)s_i$ then the result follows from  \eqref{bigrange}. If both
have $|r_i|<(n+3)s_i$ then, since $r_1r_2\equiv s_1s_2$ mod $p$ and $|r_1r_2-s_1s_2|< (n+3)^2n^2+n^2<p$, we must have $r_1r_2=s_1s_2$ and $|r_1|=s_2\leq n$. Hence $A$ has a representation $A=(tp-r)/s$, gcd$(r,s)=1$, with both $s,|r|\leq n$, and since $A$ is not of the form \eqref{critical} by assumption, we have $|r|+s>n$.
Since $p>(n^3+1)^2\geq (|r|sn+1)^2$ the result follows from  Theorem \ref{smallrs}.
\end{proof}

\section{Proof of examples}\label{ProofEx}

\begin{proof}[Proof of Example \ref{Ex2}]  (a) Suppose that $0<A<n$. Then each $Ax$, $x=1,\ldots,p-1$ will lie in $[1,A(p-1)]$ with $A(p-1)<Ap$. So reducing mod $p$ to lie in $[1,p)$ we have
$$ Ax \text{ mod } p = Ax-\ell p, \;\; 0\leq \ell\leq A-1.  $$
For $x$ in $I_i$ we have $Ax-\ell p \equiv Ai-\ell p$ mod $n$  with at most $A$ different values of $\ell,$ and so $Ax$ mod $p$  can take at most $A$ different values mod $n.$ Similarly the $-Ax$ mod $p$ take the form  $ p-(Ax -\ell p)=(\ell+1)p-Ax$, $0\leq \ell < A,$ giving at most $A$ classes mod $n$. Therefore $f(x)=Ax$ mod $p$ or $-Ax$ mod $p$ with $A<n$  must omit at least  $n-A$ classes.

(b) Suppose that $A=(tp-r)/s$ with $s>0$ and $1\leq x<p$, $\gcd(s,rt)=1$. We divide $x$ into the various residue classes mod $s$. Since $\gcd(s,t)=1$, letting $t^{-1}$ denote the mod $s$ inverse of $t$, we can write
$$x\equiv t^{-1}a  \text{ mod } s,\;\;1\leq a \leq s. $$
 Then $s\mid(ap-rx)$ and
$$ Ax \equiv \frac{ap-rx}{s} \text{ mod } p. $$
Suppose that $r>0,$ otherwise replace $A$ by $-A$ and count the $p-\ell$ mod $n$,  and set
$$ r=hs+r_0, \;1\leq r_0<s.  $$
We have
$$        \frac{ap-rx}{s} <\frac{ap}{s} \leq  p,           $$
and
$$        \frac{ap-rx}{s} >  \frac{ap-rp}{s}  = \left(-h +    \frac{a  -r_0}{s}\right)p.$$
Hence the least residue of $Ax$ mod $p$ is
$$   \frac{ap-rx}{s}  + \ell p  $$
where $\ell$ is one of the $h+1$ possibilities $0,1,\ldots ,h$ if $a\geq r_0$, or the $h+2$ possibilities
$0,1,\ldots,h,h+1$ for $1\leq a\leq r_0-1$.

Therefore, writing $m=\ell s+a$, we have  $1\leq m\leq (h+1)s+(r_0-1)=r+s-1$  and the least residues take the form
$$ \frac{mp-rx}{s},\;\;\; 1\leq m\leq r+s-1,\;\;m\equiv tx \text{ mod } s.$$

Let $b:=\gcd(n,s)$ and suppose that $x$ is in $I_i$. If $b=1$ then, for each $m,$ we have
$$\frac{mp-rx}{s} \equiv (mp-ri)s^{-1} \text{ mod } n $$
and hence obtain at most $r+s-1$ residue classes mod $n$.  If $b>1$  then $m\equiv ti$ mod $b$ and, for a given $m,$ plainly
$(mp -rx)/b\equiv (mp-ri)/b$ mod $n/b$ giving
$$\frac{mp-rx}{s}\equiv (s/b)^{-1}  (mp-ri)/b \text{  mod }n/b. $$
So we will have $b$ possible residue classes mod $n$ for each of the $m$
in  $1\leq m\leq r+s-1$  lying in a particular residue class $m\equiv ti$  mod $b$; that is, at most
\be  \label{all} b  \left\lceil \frac{r+s-1}{b}\right\rceil \leq b\left( \frac{r+s-2}{b}+1\right)=r+s+b-2 \ee
residue classes mod $n$. At least one residue class is missed when this is less than $n$.

%For $-Ax$ mod $p$  our residue classes take the form
%$$ p- \left(\frac{mp-rx}{s} \right) $$
%and the count is the same. This deals with $A=(tp+r)/s$ with $r,s>0$.

(c) We proceed as in (b). For $b=1$ there is nothing to show. So suppose that $b>1$ with  $(r+s-1)=bq+w$,
$0\leq w<b$. We take our $i$ to satisfy  $ti\equiv v$ mod $b$ for any $v$ with  $w<v\leq b$. This gives us $(n/b)(b-w)=n(1-\{\frac{r+s-1}{b}\})\geq n/b$ residue classes mod $n$. For these $i$ the number of residue classes hit  in \eqref{all} becomes
\[
b  \left\lfloor \frac{r+s-1}{b}\right\rfloor \leq r+s-1 < n.\qedhere
\]

\end{proof}

\begin{proof}[Proof of Example \ref{Ex2b}]
Recall that $Ax^{(p+1)/2}\equiv \pm Ax $ mod  $p.$  Counting the residue classes for $Ax$ or $-Ax$  mod $p$ gives
at worst twice the total obtained in the proof of  Example \ref{Ex2} for each  of these,  and therefore a missed residue class when this is less than $n$.
\end{proof}

\begin{proof}[Proof of Example \ref{Ex2c}]

(a) Suppose that $A>0$.
Notice that when $n$ is odd  or $n$ is even and $2^{\beta}\mid A$ and $x\equiv 2^{-1}p$ mod $n/\gcd(A,n)$
we have
$$Ax-\ell  p\equiv (A-\ell )p-Ax \text{ mod } n, \;\;\; \ell=0,\ldots ,A-1.$$
Thus, matching up the opposite ends  $Ax$ and $Ap-Ax,$ we can
perfectly  pair  the residue classes $Ax, Ax-p,\ldots ,Ax-(A-1)p$
for $Ax$ mod $p$ and the classes $p-Ax,2p-Ax,\ldots,Ap-Ax$ for $-Ax$  mod $p$ in reverse order.
Hence $Ax^{(p+1)/2}$ or $-Ax^{(p+1)/2}\equiv \pm Ax$ mod $p$ can take at most $A$ different values mod $n$ when $x$ is in $I_i$ for any of the $\gcd(n,A)$ values of $i$ with $i\equiv 2^{-1}p$ mod $n/\gcd(A,n)$.

(b) If $2^{\beta}\nmid A$ then we can no longer match the end values and the best we can hope for is to match up $\gcd(A,n)$ steps in. That is
$$ Ax-\gcd(A,n)p\equiv Ap-Ax \text{ mod } n, $$
so that the remaining $Ax-(\gcd(A,n)+\ell)p$ match up with the $(A-\ell)p-Ax$ mod $n.$  Thus we will just have the $Ax-\ell p$ with $0\leq \ell<\gcd(A,n)$ unmatched, and hence a total of $B:=A+\gcd(A,n)$ residue classes.  This requires $2Ax \equiv (A+\gcd(A,n))p$  mod $n,$ that is
$2A/\gcd(A,n)\;x\equiv (A/\gcd(A,n)+1)p $  mod $n/\gcd(A,n)$, equivalently $x\equiv  \frac{1}{2}(A/\gcd(A,n)+1) p (A/\gcd(A,n))^{-1} $ mod $n/2\gcd(A,n).$ Similarly we could match at the other end $p-Ax\equiv Ax- (A-1-\gcd(A,n))p$ mod $n$ for the same count. Hence if
$$i:\equiv \frac{1}{2} \left(\frac{A}{\gcd(A,n)}\pm 1\right)\left(\frac{A}{\gcd(A,n)}\right)^{-1}\hspace{-2ex }p \;\;\text{ mod } \frac{n}{\gcd(n,2A)}, $$
we have $f(I_i)\cap I_j=\emptyset$ for at least $n-B$ values of $j$.

(c), (d) and (e).  Let $b:=\gcd(n,s)$ and $c:=\gcd(n,r)$.

Suppose first  that $n$ is odd or $n$ even with $2^{\beta}\mid r$ and
$$ B:=r+s + b-2 < n. $$
Suppose that $i$ satisfies $i\equiv 2^{-1}p$ mod $n/c$.

As in the proof of Example \ref{Ex2}, for $A=(tp-r)/s$, $r,s>0$
the classes for $Ax$ mod $p$ and  $-Ax$ mod $p$ with $x$ in $I_i$  will take the form
$$ \left(\frac{ mp-rx}{s}\right) \text{ and } p-\left(\frac{ mp-rx}{s}\right)$$
respectively, with $1\leq m\leq r+s-1$, and $m\equiv tx$ mod $s$.
Writing  $m'=r+s-m$  we have
$$ p-\left( \frac{m'p-rx'}{s}\right)   = \frac{(mp-rx)}{s}+ \frac{r(x+x'-p)}{s}$$
where plainly $1\leq m\leq r+s-1$ iff $1\leq m'\leq r+s-1$ and, since  $r\equiv pt$ mod $s$,
$$ m'\equiv tx' \text{ mod } s \;\;\;\text{  iff }  \;\;\; x'\equiv p-mt^{-1} \text{ mod } s. $$
Note that when $b>1$,  the conditions $x\equiv mt^{-1}$ mod $s$ with $x$ in $I_i$ and $ m'\equiv tx' \text{ mod } s$, $x'$ in $I_i$ both imply that $m\equiv ti$ mod $b$, since $i\equiv p-i$ mod $b$.

If $b=1$ then the $x$, $x'$ in $I_i$ have $x+x'-p\equiv 2i-p\equiv 0$ mod $n/c$
and
$$  p-\left( \frac{m'p-rx'}{s}\right)   \equiv  \frac{(mp-rx)}{s} \equiv (mp-ri)s^{-1} \text{ mod } n $$
with the different $m$ only giving us  $r+s-1$ different residue classes mod $n$.

Now suppose that $b>1$ and $x,x'$ are in $I_i$, and that we have an $m$ with $1\leq m\leq r+s-1$ and $m\equiv ti$ mod $b$.
Consider the $x$ with
$$ x\equiv i \text{ mod } n/c,\;\;\; x\equiv mt^{-1} \text{ mod } s. $$
If $x_0$ is one solution then the other $x$ will satisfy $x\equiv x_0$ mod $ns/bc$.  That is, we will have $b$ solutions mod $ns/c$:
$$ x=x_0 + \lambda ns/bc  \text{ mod } ns/c,\;\;\; 0\leq \lambda < b. $$
Similarly, the
$$ x'\equiv i \text{ mod } n/c,\;\;\; x'\equiv p- mt^{-1} \text{ mod } s $$
will have $b$ solutions mod $ns/c$, namely, since $p-i\equiv i$ mod $n/c$,
$$ x'=p-x_0 - \lambda ns/bc \text{ mod } ns/c,\;\;\; 0\leq \lambda < b. $$
Thus pairing up the $x$ and $x'$  with the same $\lambda$ we get $r(x+x'-p)\equiv 0$ mod $ns$ and
$$  p-\left( \frac{m'p-rx'}{s}\right)   \equiv  \frac{(mp-rx)}{s} \text{ mod } n$$
perfectly pairing up the classes for $-Ax'$ and $Ax$. Counting the $b$ values of $\lambda$ for each $m$ with $1\leq m\leq r+s-1$ and $m\equiv ti$ mod $b$ gives the count $B$ as before and we miss $n-B$ classes.
This gives us (d), and (c) when $b=\min\{b,c\}.$

Notice that in some cases we can relax our inequality; for example if  $b>1$ but $b\mid (r+s-1)$,  or if $r$ and $\lfloor r/b\rfloor $ have opposite parity (so that if $r\equiv w$ mod $b$ then
$m\equiv 2^{-1}r \equiv \frac{1}{2}(w+b)$ mod $b$), we never have to round up  in \eqref{all}  and so only
need $r+s\leq n$.

Observe that $f(I_i)\cap I_j=\emptyset$ if and only if $f^{-1}(I_j)\cap I_i=\emptyset$ where
$$f(x)=Ax^{(p+1)/2} \text{ mod }p  \;\;  \Rightarrow  \; f^{-1}(x)=\left(\frac{A}{p}\right) A^{-1} x^{(p+1)/2} \text{ mod } p,   $$
with
$$A=(tp-r)/s  \;\; \Rightarrow\;\;  A^{-1}=(t'p-s)/r,\;\;t'\equiv sr^{-1} \text{ mod } p.$$
Switching the roles of $r$ and $s$ gives (c) when $n$ is odd and (e)
when $n$ is even.

(f) Suppose that $n$ is even $2^{\beta}\nmid r$ and that $i$ satisfies
$$i\equiv \frac{1}{2}\left((r/c)\pm 1 \right) p\left(r/c\right)^{-1} \text{ mod } n/c ,$$
(we just consider the plus sign, the case with the minus sign  is similar).
Take $m'=r+s+c-m$ and write
\begin{align*} p- & \left( \frac{m'p-rx'}{s}\right)  = \frac{mp-rx}{s}+ \frac{r(x+x'-p)-cp}{s},
\end{align*}
with $1\leq m'\leq r+s-1,$ and  hence  $1+c \leq m \leq r+s+c-1$, and
$$  x'\equiv m't^{-1} \equiv
 (r+c)t^{-1} -mt^{-1} \text{ mod } s.$$
Notice that if $x'$ is in $I_i$ then $m =s+r+c-m'\equiv r+\gcd(r,n)-ti\equiv ti$ mod $b,$ since $2it\equiv pt (r/c)^{-1}(1+(r/c)) \equiv (c+r)$ mod $b$.

Suppose that $x$, $x'$ are in $I_i$. If $b=1$ then
$$r(x+x'-p)-cp \equiv c \left(  2i(r/c)-p\left((r/c)+1\right) \right) \equiv 0 \text{ mod } n$$
and
$$ p- \left( \frac{m'p-rx'}{s}\right)  \equiv  \frac{mp-rx}{s}\equiv (mp-ri)s^{-1} \text{ mod } n. $$
For the $-Ax'$ mod $p$ we need the $1+c\leq m \leq r+s -1+c$ and for $Ax$ mod $p$ the  $1\leq m\leq r+s-1.$
Hence we have $1\leq m \leq r+s+c-1$ and at most $r+s+c-1$ residue classes mod $n$.

Suppose that $b>1$ and $m\equiv ti$ mod $b$, then taking $x_0$ to be a solution to
$$  x\equiv i \text{ mod } n/c,\;\;\; x\equiv mt^{-1} \text{ mod } s, $$
the solutions take the form
$$ x\equiv x_0 + \lambda ns/bc \text{ mod } ns/c,\;\;\; 0\leq \lambda <b. $$
Likewise, since   $(r/c)^{-1}(1+(r/c))p -i\equiv i$ mod $n/c$, the solutions to
$$  x'\equiv i \text{ mod } n/c,\;\;\; x'\equiv  (r+c)t^{-1} -mt^{-1} \text{ mod } s $$
can be written
$$ x'\equiv  (r/c)^{-1} (1+(r/c))p -x_0 - \lambda ns/bc \text{ mod } ns/c,\;\;\; 0\leq \lambda <b, $$
where here we take $(r/c)^{-1}$ to be an inverse of $r/c$ mod $ns/c$.

Pairing up the $x$ and $x'$ with the same $\lambda$ we have
 $$ p- \left( \frac{m'p-rx'}{s}\right)  \equiv  \frac{mp-rx}{s}\equiv   \frac{mp-rx_0}{s}- \lambda (r/c)(n/b)\text{ mod } n. $$
With $b$ choices of $\lambda $ for each $m\equiv ti$ mod $b$ with $1\leq m\leq r+s+c-1$ we have at most
\be \label{count}  b\left\lceil \frac{r+s+c-1}{b} \right\rceil \leq b\left( \frac{r+s+c-2}{b}+1\right) = r+s+c+b-2 \ee
residue classes mod $n$.

Notice that $ti\equiv (r+c)/2$ mod $b$
and if $b\mid (r+c)$ when $s$ is odd, or $2b \mid (r+c)$ when $s$ is even, or $b\nmid (r+c)$ and
$\lfloor (r+c)/b\rfloor$ is odd, then in (f)  we only need $r+s+c\leq n$.
Similarly when $2^{\beta}\mid s$  the value of $i$ is only fixed mod $b/2$,  hence if $r+c\equiv w$
mod $b$ we can pick an $i$ so that  $ti\equiv (w+b)/2$ mod $b$, and again we only need $r+s+c\leq n$, giving us (e) directly without flipping $r$ and $s$.
\end{proof}

\end{document}